\documentclass[12pt,letterpaper,reqno]{amsart}

\usepackage{color}

\usepackage{thmtools}
\usepackage{amsthm}

\usepackage{graphicx}
\usepackage{caption}
\usepackage{subcaption}

\graphicspath{
{../graphics/}
}

\usepackage[nospace,noadjust]{cite}

\usepackage{hyperref}
	\hypersetup{
              breaklinks,
            	linkcolor = blue,
				urlcolor = blue,
				anchorcolor = blue,
				citecolor = blue}

\usepackage{cleveref}

\numberwithin{equation}{section}
\numberwithin{figure}{section}

\newcommand{\bbC}{\mathbb{C}}
\newcommand{\bbF}{\mathbb{F}}
\newcommand{\bbQ}{\mathbb{Q}}
\newcommand{\bbR}{\mathbb{R}}
\newcommand{\bbZ}{\mathbb{Z}}

\newcommand{\cdotwild}{{}\cdot{}} 

\renewcommand{\Re}{\operatorname{Re}}
\renewcommand{\Im}{\operatorname{Im}}

\declaretheorem[numberwithin=section,name=Theorem]{thm}
\declaretheorem[sibling=thm,name=Lemma]{lem}
\declaretheorem[sibling=thm,name=Corollary]{cor}
\declaretheorem[sibling=thm,name=Conjecture]{conj}

\declaretheorem[sibling=thm,name=Definition]{defn}

\declaretheorem[sibling=thm,name=Remark]{rem}

\begin{document}

\title{Newman's conjecture in various settings}

\author[Julio Andrade]{Julio Andrade}\email{\textcolor{blue}{\href{mailto:julio_andrade@brown.edu}{julio\_andrade@brown.edu}}}

\address{Institute for Computational and Experimental Research in Mathematics (ICERM),
Brown University,
121 South Main Street
Providence, RI 02903}

\author[Alan Chang]{Alan Chang}\email{\textcolor{blue}{\href{mailto:acsix@math.princeton.edu}{acsix@math.princeton.edu}}}

\address{Department of Mathematics, Fine Hall, Washington Road, Princeton, NJ 08544}

\author{Steven J. Miller}\email{\textcolor{blue}{\href{mailto:sjm1@williams.edu}{sjm1@williams.edu}}}
\address{Department of Mathematics and Statistics, Williams College,
Williamstown, MA 01267}

\subjclass[2010]{11M20, 11M26, 11M50, 11Y35, 11Y60, 14G10}

\keywords{Newman's conjecture, zeros of the Riemann zeta function, $L$--functions, function fields, random matrix theory, Sato--Tate conjecture.}

\date{\today}

\thanks{The first and second named authors were funded by NSF Grant DMS0850577 and Williams College, and the third was partially supported by NSF Grant DMS1265673. We thank our colleagues from the 2013 Williams College SMALL REU, especially Minh-Tam Trinh, and David Geraghty and Peter Sarnak for many helpful conversations.}

\begin{abstract} De Bruijn and Newman introduced a deformation of the Riemann zeta function $\zeta(s)$, and found a real constant $\Lambda$ which encodes the movement of the zeros of $\zeta(s)$ under the deformation. The Riemann hypothesis (RH) is equivalent to $\Lambda \le 0$. Newman made the conjecture that $\Lambda \ge 0$ along with the remark that ``the new conjecture is a quantitative version of the dictum that the Riemann hypothesis, if true, is only barely so.''

Newman's conjecture is still unsolved, and previous work could only handle the Riemann zeta function and quadratic Dirichlet $L$-functions, obtaining lower bounds very close to zero (for example, for $\zeta(s)$ the bound is at least $-1.14541 \cdot 10^{-11}$, and for quadratic Dirichlet $L$-functions it is at least $-1.17 \cdot 10^{-7}$). We generalize the techniques to apply to automorphic $L$-functions as well as function field $L$-functions. We further determine the limit of these techniques by studying linear combinations of $L$-functions, proving that these methods are insufficient.

We explicitly determine the Newman constants in various function field settings, which has strong implications for Newman's quantitative version of RH. In particular, let $\mathcal D \in \bbZ[T]$ be a square-free polynomial of degree $3$. Let $D_p$ be the polynomial in $\bbF_p[T]$ obtained by reducing $\mathcal D$ modulo $p$. Then the Newman constant $\Lambda_{D_p}$ equals $\log \frac{|a_p(\mathcal D)|}{2\sqrt{p}}$; by Sato--Tate (if the curve is non-CM) there exists a sequence of primes such that $\lim_{n \to\infty} \Lambda_{D_{p_n}} = 0$. We end by discussing connections with random matrix theory.
\end{abstract}


\maketitle

\setcounter{equation}{0}

\tableofcontents

\section{Introduction}

\subsection{Newman's conjecture for the Riemann zeta function}

Let
\begin{equation}\label{eq:classical-xi}
\xi(s) \ = \  \frac{1}{2} s (s-1) \pi^{-s/2}\Gamma\left(\frac{s}{2}\right) \zeta(s)
\end{equation}
be the completed Riemann zeta function, and let
\begin{equation}
\Xi(x) \ = \  \xi\left(\frac{1}{2} + ix\right).
\end{equation}
Because of the functional equation $\xi(s) = \xi(1-s)$, we know that $x \in \bbR$ implies $\Xi(x) \in \bbR$. In general, we allow $x$ to be complex.

As $\Xi(x)$ decays rapidly as $x \to \infty$ along the real line, we have
\begin{equation}\label{eq:fourier-transform-of-xi}
\Xi(x)
\ = \
\int_{-\infty}^\infty \Phi(u) e^{i u x} \, du
\ = \
\int_{0}^\infty \Phi(u) (e^{i u x} + e^{-iux}) \, du
,
\end{equation}
where $\Phi(u) := \frac{1}{2\pi}\int_{-\infty}^\infty \Xi(x) e^{-iux} \, dx = \frac{1}{2\pi}\int_{0}^\infty \Xi(x) (e^{iix} + e^{-iux}) \, dx$ is the Fourier transform of $\Xi(x)$. In the 1920s, P\'olya introduced a ``time'' parameter $t$ to $\Xi$, given as follows:
\begin{equation}\label{eq:xi-t-x}
\Xi_t(x)
\ :=\
\int_{0}^\infty e^{tu^2} \Phi(u) (e^{i u x} + e^{-iux}) \, du
.
\end{equation}

We refer to the process beginning with \eqref{eq:classical-xi} and ending with \eqref{eq:xi-t-x} as \emph{P\'olya's setup}. For each $t \in \bbR$, \eqref{eq:xi-t-x} gives us a function in $x$ that is both $\bbC \to \bbC$ and $\bbR \to \bbR$. For $t = 0$, we recover our original function $\Xi$. The Riemann Hypothesis (RH) is equivalent to the assertion that $\Xi_0$ has only real zeros. P\'olya hoped to show that the function $\Xi_t$ has only real zeros for all $t \in \bbR$, so RH would follow.

De Bruijn and Newman proved the following results about $\Xi_t(x)$.

\begin{lem}[De Bruijn \cite{debruijn}]\label{lem:real-zeros-stay-real}
If $\Xi_t$ has only real zeros, then so does $\Xi_{t'}$ for all $t' > t$.
\end{lem}

\begin{lem}[Newman \cite{newman}]\label{lem:polya-false}
There exists some $t \in \bbR$ such that $\Xi_t$ has a non-real zero.
\end{lem}

In particular, Newman's result shows that what P\'olya had been trying to prove was actually false. However, by combining the two results of De Bruijn and Newman, we see the following.

\begin{cor}
There exists a constant $\Lambda \in \bbR$ (called the De Bruijn--Newman constant) such that

\begin{enumerate}
\item if $t \geq \Lambda$ then $\Xi_t$ has only real zeros, and
\item if $t < \Lambda$ then $\Xi_t$ has a non-real zero.
\end{enumerate}
\end{cor}

The Riemann hypothesis is equivalent to $\Lambda \leq 0$. Newman made the following complementary conjecture.

\begin{conj}[Newman's conjecture]
Let $\Lambda$ be the De Bruijn--Newman constant. Then $\Lambda \geq 0$.
\end{conj}

In Newman's own words: ``The new conjecture is a quantitative version of the dictum that the Riemann hypothesis, if true, is only barely so.''

Csordas, Smith and Varga \cite{csv} used the differential equations governing the motion of the zeros to show that ``unusually'' close pairs of zeros can give lower bounds on $\Lambda$. \cite{saouter} builds on this method of Csordas et.~al.~and achieves the current best-known lower-bound: $\Lambda \geq -1.14541 \times 10^{-11}$.

A key step in the argument of \cite{csv} is the following.

\begin{lem}[Theorem 2.2 of \cite{csv}]\label{lem:double-zero-lower-bound}
If $\Xi_{t_0}(x)$ has a zero $x_0$ of order at least $2$, then $t_0 \leq \Lambda$.
\end{lem}

\begin{rem}
Observe that if we set $F(x, t) = \Xi_t(x)$, then $F$ satisfies $\partial_t F + \partial_{xx} F = 0$, the \emph{backwards heat equation}. This PDE provides physical intuition for why \Cref{lem:double-zero-lower-bound} is true: as we \emph{decrease} $t$, the graph of $\Xi_t$ changes in accordance to the diffusion of heat. If $\Xi_{t_0}$ has a double zero, these zeros are likely to ``pop off'' the real line as we further decrease $t$. See \Cref{sec:backwards-heat-eq-example} for an example of this phenomenon.
\end{rem}

\begin{rem}
It is conjectured that all the zeros of $\zeta(s)$ are simple. If this is false, then \Cref{lem:double-zero-lower-bound} implies that Newman's conjecture is true. However, if the zeros of $\zeta(s)$ are indeed all simple, we cannot make any conclusions of Newman's conjecture. This is discussed in \Cref{rem:simplicity-no-number-field-analogue}.
\end{rem}

\subsection{Structure of this paper}

In \Cref{sec:conditions-for-generalized}, we look at conditions needed for a generalized version of Newman's conjecture. Stopple \cite{stopple} has shown that P\'olya's setup also holds for quadratic Dirichlet $L$-functions. We show it is possible to state a version of Newman's conjecture for automorphic $L$-functions. However, we only see the same behavior as before, so we quickly move on to rational function fields $\bbF_q(T)$.

As in the number field case, each quadratic Dirichlet $L$-function $L(s, \chi_D)$ in the function field setting also gives rise to a constant $\Lambda_D$. This case, which we look at in \Cref{sec:fn-fields} (the main section of the paper), exhibits \emph{very} different behavior. First of all, RH is true, so we know $\Lambda_D \leq 0$. Second, the statement of Newman's conjecture is different.

Whereas Newman's conjecture in number fields is $\Lambda \geq 0$, it is not necessarily true that $\Lambda_D \geq 0$ in the function field setting. 
In fact, we can have $\Lambda_D = -\infty$.
However, if we look at certain ``families'' $\mathcal F$ of $L$-functions (as discussed in \Cref{subsec:newman-families}), we have reason to believe that the supremum of $\Lambda_D$ over these family is nonnegative.

\begin{conj}[Newman's conjecture in the function fields setting]
Let $\mathcal F$ be a family of $L$-functions over a function field. Then
\begin{equation}
\sup_{D \in \mathcal F}
\Lambda_D
= 0
.
\end{equation}
\end{conj}

For an example of a family, suppose we fix an elliptic curve $y^2 = \mathcal D(T)$ over $\bbQ$, and look at the polynomials $D_p \in \bbF_p[T]$ obtained by reducing $\mathcal D$ modulo $p$. Let $a_p(\mathcal D)$ be the trace of Frobenius of the elliptic curve $y^2 = \mathcal D(T)$. In \Cref{subsec:deg-3-case}, we prove that Newman's conjecture is true for this family, and explicitly relate the Newman constant to $a_p(\mathcal D)$.

\begin{thm}[Newman's conjecture for fixed $\mathcal D$, $\deg \mathcal D = 3$]
Let $\mathcal D \in \bbZ[T]$ be a square-free polynomial of degree $3$. Let $D_p$ be the polynomial in $\bbF_p[T]$ obtained by reducing $\mathcal D$ modulo $p$. Then \begin{equation}
\Lambda_{D_p} \ = \  \log \frac{|a_p(\mathcal D)|}{2\sqrt{p}},
\end{equation}
which implies $\sup_{p} \Lambda_{D_p} = 0$.
\end{thm}

To show the supremum is zero, we use the recent proof of the Sato--Tate conjecture \cite{B-LGHT, satotate1, satotate3, satotate2}. This implies that the Newman conjecture for function fields is connected to deep results in number theory.

Next, we change our approach to Newman's conjecture in function fields and use results from random matrix theory to support our conjecture. By relating random matrix theory statistics to the distributions of the zeros of our functions $\Xi(x, \chi_D)$, we prove Newman's conjecture for a different family. For detailed statements, see \Cref{subsec:lowlyingzerosandconnections}.

Finally, in \Cref{sec:numerical} we examine the results of some numerical computations. In particular, we observe that as we increase the degree, we find elements $D\in\bbF_3[T]$ such that the Newman constants $\Lambda_D$ approach zero.

\section{Conditions for a generalized Newman's conjecture}\label{sec:conditions-for-generalized}

\subsection{Stopple's generalization of Newman's conjecture}

Stopple \cite{stopple} showed that if $D$ is a fundamental discriminant and $\chi_D(n)$ is the Kronecker symbol $(\frac{D}{n})$, then we can apply P\'olya's setup for $\zeta(s)$ to the Dirichlet $L$-function $L(s, \chi_D)$. This gives us an analogue of \eqref{eq:xi-t-x}:
\begin{equation}\label{eq:xi-t-chi-D-for-number-field}
\Xi_t(x, \chi_D)
\ :=\
\int_{0}^\infty e^{tu^2} \Phi(u, \chi_D) (e^{i u x} + e^{-iux}) \, du.
\end{equation}
Each $D$ has its own De Bruijn--Newman constant $\Lambda_D$, and most of the techniques in \cite{csv} for attaining lower bounds on $\Lambda$ carry over to $\Lambda_D$.

\begin{conj}[Newman's conjecture for quadratic Dirichlet $L$-functions]\label{conj:newman-dirichlet-l-functions}
Let $D \in \bbZ$ be a fundamental discriminant. Then $\Lambda_D \geq 0$.
\end{conj}

Stopple investigated a variation of this conjecture.

\begin{conj}[Newman's conjecture for quadratic Dirichlet $L$-functions, weaker form]\label{conj:newman-dirichlet-l-functions-weak}
 We have $\sup_D \Lambda_D \geq 0$, where the supremum is taken over all fundamental discriminants $D$.
\end{conj}

Note that \Cref{conj:newman-dirichlet-l-functions} implies \Cref{conj:newman-dirichlet-l-functions-weak}. Instead of looking for close pairs of zeros along the real line, Stopple looks for $L$-functions $L(s, \chi_D)$ with ``unusually'' low-lying zeros.%
\footnote{Let the positive zeros of $\Xi(x, \chi_D)$ be denoted $\gamma_1, \gamma_2, \ldots$ with $0 < \gamma_1 < \gamma_2 < \cdots$. Define \[ G\ = \ \sum_{|j| \geq 2}  \left[ \frac{1}{(-\gamma_{1} - \gamma_j)^2} + \frac{1}{(\gamma_{1} - \gamma_j)^2}  \right]. \] Then the zero $\gamma_1$ is ``unusually'' low-lying, in the sense given in \cite{stopple}, if $5 \gamma_1^2 G < 1$. Stopple calls such $D$ ``Low discriminants'' (``Low'' is a person's name).} If an $L$-function has an unusually low-lying zero $\gamma$, then the zeros $\pm \gamma$ would then form a close pair.

Stopple found that for $D = -175990483$, we have $-1.13 \cdot 10^{-7} < \Lambda_D$.

\subsection{Sufficient conditions for generalization}\label{sec:conditions-for-generalization}

Let $L(s, f)$ be the $L$-function associated with some object $f$. In accordance with notation introduced earlier, let $\xi(s, f)$ be the completed $L$-function and let $\Xi(x, f) = \xi(\frac{1}{2} + ix, f)$.

If we try to define $\Xi_t(x, f)$ analogously, we need the following.

\begin{enumerate}
\item
$\Xi(\cdotwild, f)$ has to restrict to a $\bbR \to \bbR$ function, so that we can define the Fourier transform $\Phi(u, f)$, as in \eqref{eq:fourier-transform-of-xi}. It is sufficient to have the functional equation $\xi(s, f) = \xi(1-s,f)$.

\item
$\Phi(u, f)$ has to have extremely rapid decay in order for the integral in \eqref{eq:xi-t-x} to converge for each $t \in \bbR$. It is sufficient to have $\Phi(u, f) = O(e^{-|u|^{2+\epsilon}})$ for some $\epsilon > 0$.
\end{enumerate}

Usually, the rapid decay of $\Phi(u, f)$ can be seen because it has an infinite sum representation. For instance, in the case of the Riemann zeta function, we have
\begin{equation}
\Phi(u)
\ = \
2
\sum_{n=1}^\infty
(2n^4 \pi^2 e^{9u/2} - 3n^2 \pi e^{5u/2})
e^{-n^2 \pi e^{2u}}
,
\end{equation}
which shows that $\Phi(u)$ decays as quickly as a double exponential. See \cite[(10.1.4)]{titchmarsh}.

\subsection{Slight modifications of P\'olya's setup and limitations}

The most stringent requirement is the even functional equation: $\xi(s, f) = \xi(1-s,f)$. This is why Stopple did not investigate \emph{all} Dirichlet $L$-functions -- the complex $L$-functions do not have the proper type of symmetry.

In general, a completed $L$-function satisfies a functional equation of the form $\xi(s, f) = \epsilon_f \xi(1-s, \overline f)$, where $|\epsilon_f| = 1$. If we define $\Xi(x, f) = \xi(\frac{1}{2} + ix, f)$, then $\Xi(x, f) = \epsilon_f \Xi(-x, \overline f)$. Thus, we need $f$ to be self-dual (i.e., $f = \overline f$) and we need the root number $\epsilon_f$ to be $1$.

There are two straightforward attempts to ``fix'' a bad functional equation, but they both fail when we attempt to state Newman's conjecture for the $L$-function.

\begin{enumerate}
\item
We replace $\xi(s, f)$ with $\tilde \xi(s, f) := |\xi(s, f)|$. Then $\tilde \xi(s, f) = \tilde \xi(1-s, f)$ for any $L$-function. However, $\tilde \xi(s, f)$ is no longer a smooth function in $s$. Thus, we lose the backwards heat equation and other desirable properties and the results of \cite{csv} do not carry over.
\item
We replace $\xi(s, f)$ with $\tilde \xi(s, f) := |\xi(s, f)|^2$, which is smooth in $s$. In this case we have an analogue of \Cref{lem:double-zero-lower-bound}, but since every zero of $\tilde \xi(s, f)$ is doubled, the lemma would make Newman's conjecture for $\tilde \xi(s, f)$ trivially true.
\end{enumerate}

If we have an $L$-function with odd functional equation $\xi(s, f) = -\xi(1-s, f)$, what we \emph{can} do is define $\tilde \xi(s, f) = \frac{i}{s-1/2} \xi(s, f)$, which will then satisfy the conditions in \Cref{sec:conditions-for-generalization}.

Alternatively, we can consider products of different $L$-functions. For example, if we have two odd $L$-functions $\xi(s, f)$ and $\xi(s, g)$, then the product $\tilde\xi(s) := \xi(s, f)\xi(s, g)$ has the desired even symmetry. Thus there is a P\'olya setup for $\tilde \xi$, and a corresponding constant $\Lambda$. In this case, it is the distribution of union of the zeros of each $L$-function that become relevant. (If the two $L$-functions share a zero then we have a double zero and Newman's conjecture is trivially true.)


Because of the lack of a proper functional equation, we cannot generalize P\'olya's setup to (for example) the Hurwitz zeta function or linear combinations of $L$-functions.

\subsection{Automorphic $L$-functions}

One class of examples which can be analyzed with these methods is $H_k^+(N)$, the holomorphic cuspidal newforms of weight $k$ and level $N$ with even functional equation.

\begin{lem}
As in \cite[Section 3]{ils}, consider the Hecke $L$-function given by $L(s, f) = \sum_{n \geq 1} \lambda_f(n) n^{-s}$ for $f \in H_k^+(N)$, and let $\Xi(x, f)$ be the completed $L$-function evaluated at $s = \frac{1}{2} + ix$. Then we can follow P\'olya's setup and introduce the analogous deformation $\Xi_t(x, f)$, so there is a De Bruijn--Newman constant $\Lambda_f$ for each $f \in H_k^+(N)$.
\end{lem}

\begin{proof}
By definition of $H_k^+(N)$, the $L$-functions have even symmetry. Also, we have
\begin{equation}
\Xi(x, f)
\ = \
\int_0^\infty \Phi(u) (e^{iux} + e^{-iux}) \, du,
\end{equation}
where
\begin{equation}
\Phi(u, f)
\ = \
\left(
 \frac{2\pi}{\sqrt{N}}
\right)^{(k-1)/2}
\sum_{n \geq 1}
\lambda_f(n) n^{(k-1)/2}
\exp
\left(
 -\frac{2\pi n}{\sqrt{N}} e^{u} + \frac{k}{2} u
\right)
,
\end{equation}
which shows that $\Phi(u)$ decays rapidly as $u \to \infty$. Thus, both conditions described in \Cref{sec:conditions-for-generalization} are satisfied.
\end{proof}

\begin{conj}[The generalized Newman's conjecture for $H_k^+(N)$]
Let $f \in H_k^+(N)$. Then $\Lambda_f \geq 0$.
\end{conj}

In fact, most of the results in \cite{csv} and \cite{stopple} on lower bounds of $\Lambda_f$ carry over. However, while we are able to make a Newman's conjecture in the automorphic forms setting, we see only the same behavior as before. Thus in the next section we focus our attention on function field $L$-functions, where many new phenomena appear.

\section{Newman's conjecture for function fields}\label{sec:fn-fields}

\subsection{Background on function fields}

In the function fields setting, the appropriate analogue of $\bbZ$ is $\bbF_q[T]$, the coordinate ring of the affine line over $\bbF_q$. The background introduced here is given in more detail in \cite[Section 2]{rudnick} or \cite[Section 3]{andrade}. For a comprehensive text on number theory in function fields, see \cite{rosen}.

\begin{defn}
Let $q$ be an odd prime power and let $D \in \bbF_q[T]$. For this paper, we will say that $(q, D)$ is a \emph{good pair} if

\begin{itemize}
\item $D$ is square-free and monic,
\item $\deg D$ is odd,
\item $\deg D \geq 3$.
\end{itemize}
\end{defn}

For $(q, D)$ a good pair, let $\chi_D : \bbF_q[T] \to \{-1, 0, 1\}$ be the quadratic character modulo $D$. That is, $\chi_D(f) = ( \frac{D}{f} )$, where $(\frac{\,\cdot\,}{\,\cdot\,})$ is the Kronecker symbol.

\begin{rem}
We assume $q$ is odd because if a field has characteristic $2$, then every element is a perfect square. We assume $D$ is square-free and monic because this corresponds to the fundamental discriminants in the number field setting.

We assume $\deg D$ is odd only for simplicity and ease of exposition. The case when $\deg D$ is even can be handled similarly with some modifications.
\end{rem}

For $(q, D)$ a good pair, we define the $L$-function
\begin{equation}
L(s, \chi_D)\ :=\ \sum_{f \text{ monic}} \frac{\chi_D(f)}{|f|^s}
.
\end{equation}
By collecting terms, we can write
\begin{equation}\label{eq:l-fn-sum-over-q-s}
L(s, \chi_D)
\ = \
\sum_{n=0}^\infty
c_n
(q^{-s})^n,
\end{equation}
where
\begin{equation}\label{eq:l-fn-coefficients}
c_n
\ = \
\sum_{\substack{f \text{ monic} \\ \deg f = n}}
\chi_D(f)
.
\end{equation}

It can be shown that $c_n = 0$ for all $n \geq \deg D$, so $L(s, \chi_D)$ is a polynomial in $q^{-s}$ of at most degree $\deg D-1$. In fact, the degree is exactly $\deg D - 1$.

Let $g = \frac{1}{2}(\deg D - 1)$; we use the letter $g$ because the value of $g$ is the genus of the hyperelliptic curve $y^2 = D(T)$. The completed $L$-function $\xi(s, \chi_D) := q^{gs}L(s, \chi_D)$ satisfies the functional equation $\xi(s, \chi_D) = \xi(1-s, \chi_D)$. Note that this satisfies the symmetry type we need, as discussed in \Cref{sec:conditions-for-generalized}.

\begin{rem}\label{rem:RH-curves}
By the Riemann Hypothesis for curves over a finite field, proved by Andre Weil, we know that all the zeros of $L(s, \chi_D)$ lie on the critical line $\Re(s) = \frac{1}{2}$. (A detailed exposition of Bombieri's proof is given in \cite[Appendix]{rosen}.)
\end{rem}

Using the functional equation, we can write
\begin{equation}\label{eq:xi-function-definition}
\Xi(x, \chi_D)
\ :=\
\Lambda\left(\frac{1}{2} + i \frac{x}{\log q}, \chi_D\right)
\ = \
\Phi_0 + \sum_{n = 1}^g \Phi_n \cdot ( e^{inx} + e^{-inx} )
\end{equation}
for some constants
\begin{equation}\label{eq:fourier-coefficients}
\Phi_n \ = \  c_{g-n}q^{n/2} \ = \  c_{g+n} q^{-n/2}
;
\end{equation}
the two equivalent expressions for $\Phi_n$ are due to the symmetry of the completed $L$-function.

\subsection{Introducing the $t$ parameter}

We observe that the right side of \eqref{eq:xi-function-definition} gives the Fourier series of our completed $L$-function. We can introduce a new parameter as in \eqref{eq:xi-t-chi-D-for-number-field}, and find
\begin{equation}
\Xi_t(x, \chi_D)
\ :=\
\Phi_0 + \sum_{n = 1}^g \Phi_n e^{t n^2} ( e^{inx} + e^{-inx} ).
\end{equation}

\begin{rem}
What we call $\Phi_n$ here is the analogue of $\Phi(u)$ defined in the number field setting. In both cases, $\Phi$ is the Fourier transform of $\Xi$. The difference is that in the number field setting, $\Xi(x)$ is a function on $\bbR$ with rapid decay as $|x| \to \infty$, whereas here in the function field $\Xi(x)$ is now defined on the circle ($x \in [0, 2\pi]$). This is the reason that $\Phi$ is now a $\bbZ \to \bbR$ function.
\end{rem}

In order to guarantee the existence of a De Bruijn--Newman constant $\Lambda_D$, we need the following analogue of \Cref{lem:real-zeros-stay-real}.

\begin{lem}\label{lem:real-zeros-stay-real-fnfield}
Suppose for some $t$ that $\Xi_t(x, \chi_D)$ has only real zeros. Then for all $t' > t$, $\Xi_{t'}$ has only real zeros.
\end{lem}

\Cref{lem:real-zeros-stay-real-fnfield} immediately follows from the following lemma.

\begin{lem}[Analogue of Theorem 13 in \cite{debruijn}]\label{lem:de-bruijn-analogue}
Suppose $F : \bbZ \to \bbC$ satisfies $\sum |F(n)| < \infty$, $F(n) = \overline{F(-n)}$ and $F(n) = O(e^{-|u|^{2+\epsilon}})$ for some $\epsilon > 0$. Also suppose that the roots of $\sum_{n=-\infty}^\infty F(n) e^{inx}$ satisfy $|\Im z| \leq \Delta$ for some $\Delta \geq 0$. Then all the roots of $\sum_{n=-\infty}^\infty F(n) e^{tn^2} e^{inx}$ lie in the strip $| \Im z | \leq \max(\Delta^2 - 2t, 0)^{1/2}$.
\end{lem}

\begin{proof}
The key idea is to take (3.6) in De Bruijn's paper, which is the trigonometric integral $f(z)$ $=$ $\int_{-\infty}^\infty F(t) e^{izt} dt$, and replace it with the trigonometric sum $f(z)$ $=$ $\sum_{n=-\infty}^\infty F(n) e^{inx}$. Then we note that the arguments to Theorems 11, 12, and 13 in De Bruijn's paper can be generalized to this situation.
\end{proof}

\begin{proof}[Proof of \Cref{lem:real-zeros-stay-real-fnfield}]
Let
\begin{equation}
F(n)
\ = \
\begin{cases}
\Phi_{|n|} &\text{if } |n| \leq g
\\
0 &\text{if } |n| > g
,
\end{cases}
\end{equation}
and apply \Cref{lem:de-bruijn-analogue}.
\end{proof}

\begin{rem}
\Cref{lem:real-zeros-stay-real-fnfield} can be phrased as ``zeros on the real line remain on the real line.'' \Cref{lem:de-bruijn-analogue} gives us more than that. It also tells us that the zeros off the real line move towards the line, and furthermore provides a lower bound on the speed at which the zeros move.

For instance, if we know all the zeros of $\Xi_{t_0}(x, \chi_D)$ lie in the strip $|\Im x| \leq \Delta$, then we know all the zeros are real by the time $t = t_0 + \frac{1}{2}\Delta^2$. In the number field case, since we know the zeros of $\Xi(x)$ (for the Riemann zeta function) lie in the critical strip $|\Im(x)| \leq \frac{1}{2}$, we know that $\Lambda \leq \frac{1}{2}$.
\end{rem}

By \Cref{lem:real-zeros-stay-real-fnfield} and RH (see \Cref{rem:RH-curves}), we know that for each good pair $(q, D)$, there exists a constant $\Lambda_D \in [-\infty, 0]$ such that

\begin{enumerate}
\item if $t \geq \Lambda$ then $\Xi_t(\cdotwild, \chi_D)$ has only real zeros, and
\item if $t < \Lambda$ then $\Xi_t(\cdotwild, \chi_D)$ has a non-real zero.
\end{enumerate}

Note that we have not eliminated the possibility of $\Lambda_D = -\infty$. However, it turns out that the analogue of \Cref{lem:polya-false} is false in the function field setting; that is, there are $L$-functions with the property that $\Xi_t( \cdotwild, \chi_D)$ has only real zeros for \emph{all} $t$. (\Cref{rem:example-with-minus-infinity} contains an example.)

There is a partial analogue, which holds for irreducible $D$. This at least assures us that $\Lambda_D \neq -\infty$ often.

\begin{lem}\label{lem:polya-false-irreducibles}
Let $(q, D)$ be a good pair and suppose $D$ is irreducible. Then there exists some $t \in \bbR$ such that $\Xi_t$ has a non-real zero.
\end{lem}

\begin{proof}
First we show that $\Phi_n \neq 0$ for all $0 \leq n \leq g$. Using \eqref{eq:l-fn-coefficients} and \eqref{eq:fourier-coefficients}, we have
\begin{equation}
\frac{\Phi_n}{q^{n/2}}
\ = \
c_{g-n}
\ = \
\sum_{\substack{f \text{ monic} \\ \deg f \ = \  g-n}}
\chi_D(f)
.
\end{equation}
Since $q$ is odd, the number of terms in the sum is odd. Every $f$ in the sum is relatively prime to $D$, since $g-n < 2g+1 = \deg D$. Hence, every term in the sum is either $+1$ or $-1$. Thus $c_{g-n}$ is odd, so $\Phi_n \neq 0$.

Using the fact that $\Phi_n \neq n$, we can complete the proof of the lemma. For very negative $t$ (i.e., as $t \to -\infty$), the main terms of $\Xi_t(x, \chi_D)$ are $\Phi_0 + \Phi_1 e^{t} (e^{ix} + e^{-ix})$. If $x$ is a zero, we have
\begin{equation}
|\Phi_0/\Phi_1|e^{|t|} \approx |e^{ix} + e^{-ix}|.
\end{equation}
As $t \to -\infty$, the left side goes to $\infty$ (since $\Phi_0 \neq 0$), so for some $t$, the left side exceeds $2$, which means $x$ cannot be real.
\end{proof}

\begin{rem}
We can see from the proof of \Cref{lem:polya-false-irreducibles} that the conclusion of the lemma holds if at least two of the Fourier coefficients of $\Xi_t(x, \chi_D)$ are nonzero.
\end{rem}

\begin{rem}\label{rem:example-with-minus-infinity}
An example of an $L$-function with $\Lambda_D = -\infty$ is $D = T^3 + T \in \bbF_3[T]$. For this polynomial, $\Xi_t(x, \chi_D) = \sqrt{3} e^t \cos x$. As expected, $D$ is not irreducible -- we have $D = T(T^2+1)$.
\end{rem}





\subsection{The failure of Newman's conjecture for individual $L$-functions}

Using the results of the previous section, we know that for each $L$-function $L(s, \chi_D)$, there is a De Bruijn--Newman constant $\Lambda_D$.

At first, the ``obvious'' generalization of Newman's conjecture to this setting is that $\Lambda_D \geq 0$ for each $D$. However, this is false. \Cref{rem:example-with-minus-infinity} provides an example with $\Lambda_D = -\infty$. \Cref{sec:backwards-heat-eq-example} provides an example with $-\infty < \Lambda_D < 0$. In fact, for most (if not all) $D$, $\Lambda_D$ will be strictly negative.

\begin{lem}\label{lem:simple-implies-lambda-negative}
Let $(q, D)$ be a good pair. If $\Xi_0(x, \chi_D)$ has only simple zeros, then $\Lambda_D < 0$.
\end{lem}

\begin{proof}
The two key ideas of this technical argument are to use the implicit function theorem, and to note that there are only finitely many zeros (which is very different than the number field cases). Write $F(x, t) = \Xi_t(x)$. Suppose $\gamma$ is a simple zero of $\Xi_0(x, \chi_D)$. By the implicit function theorem, we can find a time interval $(-\epsilon, 0]$ and a function $y : (-\epsilon, 0] \to \bbR$ defined on this time interval such that $y(0) = \gamma$ and $F(y(s), s) = 0$ for all $s \in (-\epsilon, 0]$.

Because $\Phi_g \neq 0$, we know that $\Xi_0(x, \chi_D)$ has exactly $2g$ zeros (with multiplicity) in a period $0 \leq \Re(x) < 2 \pi$. Suppose all these zeros are simple, so we can write them as $0 < \gamma_1 < \gamma_2 < \cdots < \gamma_{2g} < 2\pi$. (We know the zeros of $\Xi_0$ are real because of \Cref{rem:RH-curves}.)

For each zero, there is a time interval $(-\epsilon_n, 0]$ such that the zero $\gamma_n$ stays real in this interval. Then all the zeros stay real inside the time interval $(-\epsilon, 0]$, where $\epsilon = \min\{\epsilon_1, \ldots, \epsilon_{2g}\}$. Finally, since $\Xi_t(x, \chi_D)$ has exactly $2g$ zeros in $0 \leq \Re(x) < 2\pi$ for every $t$, we have accounted for all of them.
\end{proof}

\begin{rem}
It is not known whether there exists a good pair $(q, D)$ such that $\Xi_0(x, \chi_D)$ has a double zero.
\end{rem}

\begin{rem}\label{rem:simplicity-no-number-field-analogue}
There is no analogue of \Cref{lem:simple-implies-lambda-negative} in the number field setting. A crucial part of the argument is the periodicity of $\Xi_t(x, \chi_D)$. Thus, $\epsilon$ is the minimum of a finite set of positive numbers (as opposed to the infimum of an infinite set), which allows us to conclude that $\epsilon$ is strictly positive.
\end{rem}

\subsection{Newman's conjecture for families of $L$-functions}\label{subsec:newman-families}

Because of \Cref{lem:simple-implies-lambda-negative}, we do not look at individual $L$-functions. Instead, following Stopple, we study families of $L$-functions.

\begin{conj}[Newman's conjecture, fixed $q$]\label{conj:newman-fixed-q}
Keep $q$, the number of elements of the finite field, fixed. Then
\begin{equation}
\sup_{(q, D) {\rm\ good\ pair}}
\Lambda_D
\geq 0
.
\end{equation}
\end{conj}

\begin{conj}[Newman's conjecture, fixed $g$]\label{conj:newman-fixed-g}
Keep $g$, the genus, fixed. Then
\begin{equation}
\sup_{\substack{\deg D \ = \  2g+1\\(q, D) {\rm\ good\ pair}}} \Lambda_D
\geq 0
.
\end{equation}
\end{conj}

\begin{conj}[Newman's conjecture, fixed $\mathcal D$]\label{conj:newman-fixed-D}
Fix $\mathcal D \in \bbZ[T]$ square-free. For each prime $p$, let $D_p$ be the polynomial in $\bbF_p[T]$ obtained by reducing $\mathcal D$ modulo $p$. Then
\begin{equation}
\sup_{(p, D_p) {\rm\ good\ pair}} \Lambda_{D_p}
\geq 0
,
\end{equation}
%
\end{conj}

\begin{rem}
As RH has been proved in this setting in the conjectures above, we could replace the greater than or equal to $0$ with equal to $0$; we wrote it as above to remind the reader of the analogues of Newman's conjecture in the number field setting.
\end{rem}

More generally, let $\mathcal F$ be a set of polynomials $D$ belonging to good pairs $(q, D)$ (where $q$ can vary). The corresponding family of $L$-functions is $\{ L(s, \chi_D) : D \in \mathcal F\}$. (We often use $\mathcal F$ to refer to not only the family of polynomials but also the family of $L$-functions.) The Newman's conjecture for a family $\mathcal F$ is the statement that
\begin{equation}
\sup_{D \in \mathcal F} \Lambda_D \geq 0
.
\end{equation}

The families $\mathcal F$ corresponding to \Cref{conj:newman-fixed-q}, \Cref{conj:newman-fixed-g} and \Cref{conj:newman-fixed-D}, respectively, are
\begin{itemize}
\item
Fix $q$ and let $\mathcal F = \{ D : (q, D) \text{ is a good pair} \}$.
\item
Fix $g$ and let $\mathcal F = \{ D : (q, D) \text{ is a good pair}, \deg{D} = 2g+1 \}$.
\item
Fix $\mathcal D \in \bbZ[T]$ square-free and let $\mathcal F = \{ D_p : (p, D_p) \text{ is a good pair}\}$.
\end{itemize}

\subsection{The case $\deg \mathcal D = 3$ and the Sato--Tate Conjecture}\label{subsec:deg-3-case}

We examine a special case of \Cref{conj:newman-fixed-D} in which the fixed square-free polynomial $\mathcal D \in \bbZ[T]$ satisfies $\deg \mathcal D = 3$, so $g = 1$. In this section we prove this special case of Newman's conjecture.

For a fixed $\mathcal D$ of degree $3$, the corresponding $\Xi$ functions have the form
\begin{equation}
\Xi_t(x, \chi_{D_p}) \ = \  -a_p(\mathcal D) + 2 \sqrt{q} e^{t} \cos x,
\end{equation}
where
\begin{equation}
a_p(\mathcal D)
 \ = \
\sum_{
 \substack{f \in \bbF_p[T] \\ \deg f = 1\\ f \text{ monic}}
}
\chi_{D_p}(f)
.
\end{equation}
Note that $a_p(\mathcal D)$ is the trace of Frobenius of the elliptic curve $y^2 = \mathcal D(T)$. In this setting, we get an explicit formula for $\Lambda_{D_p}$.

\begin{lem}
Let $\mathcal D \in \bbZ[T]$ be a square-free polynomial of degree $3$. Let $D_p$ be the polynomial in $\bbF_p[T]$ obtained by reducing $\mathcal D$ modulo $p$. Then
\begin{equation}
\Lambda_{D_p} \ = \  \log \frac{|a_p(\mathcal D)|}{2\sqrt{p}}.
\end{equation}
\end{lem}

\begin{proof}
Fix $t$ and suppose $x_0$ is a zero of $\Xi_t$. Then
\begin{equation}
\cos x_0
\ = \
\frac{1}{ e^t} \cdot \frac{2 \sqrt{p}}{a_p(\mathcal D)}.
\end{equation}
If $e^t \geq \frac{|a_p(\mathcal D)|}{2\sqrt{p}} $, then $-1 \leq \cos x_0 \leq 1$, so $x_0$ is real. On the other hand, if $e^{t} < \frac{|a(\mathcal D)|}{2\sqrt{q}}$, then $|\cos x_0| > 1$, which implies $x_0$ is not real.
\end{proof}

In order to show that $\sup_p \Lambda_{D_p} = 0$, we need a sequence of primes $p_1, p_2, \ldots$ such that
\begin{equation}\label{eq:seq-primes-to-1}
\lim_{n \to \infty}
\frac{a_{p_n}(\mathcal D)}{2\sqrt{p_n}}
\to
1
.
\end{equation}
Thus we need to investigate the distribution of $\frac{|a_p(\mathcal D)|}{2\sqrt{p}}$ as $p$ varies. Hasse showed in the 1930s that $-1 < \frac{a_p(\mathcal D)}{2\sqrt{p}} < 1$ \cite[Theorem V.1.1]{silverman}. (In fact, this is a special case of RH for curves over a finite field.)

It is natural to let $\theta_p \in (0, \pi)$  satisfy
\begin{equation}
\cos \theta_p
\ = \
\frac{a_p(D)}{2\sqrt{p}}
\end{equation}
and to study the distribution of $\theta_p$ as $p$ ranges.

If $\mathcal D$ has complex multiplication, then there is a spike at $\theta_p = \pi$ and otherwise a uniform distribution on $(0, \pi)$.

If $\mathcal D$ does not have complex multiplication, then the distribution is conjectured to satisfy the semi-circle distribution:
\begin{equation}
\lim_{N \to \infty}
\frac{
\# \{ p \leq N : \alpha < \theta_p < \beta \}
}{
\# \{ p \leq N \}
}
\ = \
\frac{2}{\pi} \int_\alpha^\beta \sin^2 \theta \, d\theta.
\end{equation}
This is a specific case of the Sato--Tate conjecture. Clozel, Harris, Shepherd-Barron and Taylor \cite{satotate1, satotate3, satotate2} proved this for elliptic curves without complex multiplication, provided that there is at least one prime of multiplicative reduction; that assumption was recently removed by  Barnet-Lamb, Geraghty, Harris and Taylor \cite{B-LGHT}. Thus, for any square-free $\mathcal D \in \bbZ[T]$, we can find a sequence of primes satisfying \eqref{eq:seq-primes-to-1}. We have therefore proved the following.

\begin{thm}[Newman's conjecture for fixed $\mathcal D$, $\deg \mathcal D = 3$]
Let $\mathcal D \in \bbZ[T]$ be square-free with $\deg \mathcal D = 3$. Then $\sup_{p} \Lambda_{D_p} = 0$.
\end{thm}

\begin{rem}
When $g \geq 2$, then $\Xi_t$ contains multiple $e^{t}$ terms and multiple $\cos nx$ terms, making it much harder to find the explicit expression of $\Lambda_{D_p}$.
\end{rem}

\subsection{Zeros of $\Xi_0(x, \chi_D)$}\label{sec:notation-for-zeros}

In this section, we introduce notation for the zeros of $\Xi(x, \chi_D)$ and discuss basic properties, which will be used in the remainder of the paper.

\begin{rem}
Because of \eqref{eq:xi-function-definition}, a zero $\gamma$ of $\Xi(x, \chi_D)$ corresponds to a zero $\frac{1}{2} + \frac{i\gamma}{\log q}$ of $L(s, \chi_D)$.
\end{rem}

The following analogue of \Cref{lem:double-zero-lower-bound} gives us a lower bound on $\Lambda_D$ via double zeros.

\begin{lem}\label{lem:double-zero-lower-bound-fnfield}
Let $(q, D)$ be a good pair. Let $t_0 \in \bbR$. If $\Xi_{t_0}(x, \chi_D)$ has a zero $x_0$ of order at least $2$, then $t_0 \leq \Lambda_D$.
\end{lem}

\begin{proof}
If $F(x, t) = \Xi_t(x, \chi_D)$, then $F$ satisfies the backwards heat equation: $\partial_{t} F + \partial_{xx} F = 0$. Using the observation, the lemma follows via the argument given in \cite[Lemma, page 7]{stopple}.
\end{proof}

\begin{rem}
Because of \Cref{lem:double-zero-lower-bound-fnfield}, if $\Xi_0(x, \chi_D)$ has a double zero, then Newman's conjecture is true. For most of the remaining paper, we assume that all the zeros of $\Xi_0$ are simple.
\end{rem}

Let $(q, D)$ be a good pair and assume the zeros of $\Xi_0(x, \chi_D)$ are simple. Because of evenness, this implies that $\Xi_0$ does not have a zero at $x = 0$. Let the positive zeros of $\Xi_0(x, \chi_D)$ be denoted $\gamma_1, \gamma_2, \ldots$, counted with multiplicity. We assume the zeros are ordered so that $0 < \gamma_1 < \gamma_2 < \cdots$.

By \eqref{eq:xi-function-definition}, we see that the first $2g$ zeros lie in the interval $(0, 2\pi)$, and the remaining zeros are repeated by periodicity. Thus, all the zeros of $\Xi$ are given by
\begin{equation}
\{
\gamma_j + 2 \pi \ell
:
j \in \{1, 2, \ldots, 2g \}, \ell \in \bbZ
\}
.
\end{equation}
Next, by evenness and periodicity, we know for $1 \leq j \leq g$, we have $\gamma_{2g+1-j} = - \gamma_j + 2\pi$. This implies that the first $g$ zeros lie in $(0, \pi)$ and the next $g$ zeros lie in $(\pi, 2\pi)$. Thus, all the zeros of $\Xi_0$ are given by
\begin{equation}\label{eq:all-zeros-gamma}
\{
\epsilon \gamma_j + 2 \pi \ell
:
\epsilon \in \{ \pm 1\} , j \in \{1, 2, \ldots, g \}, \ell \in \bbZ
\}
.
\end{equation}
In other words, once we compute the first $g$ zeros of $\Xi_0$, we know the remaining zeros.

\begin{rem}
The observations above still apply if $\Xi_0(x, \chi_D)$ does not have only simple zeros. The only technical detail we have to pay attention to is if $\Xi_0$ has a zero at $x = 0$. We know that $\Xi_0$ has a zero of even order, say $2n$. Then we must let $0 = \gamma_1 = \cdots = \gamma_n < \gamma_{n+1}$, so that $-\gamma_1, \ldots, -\gamma_n$ cover the remaining multiplicities.
\end{rem}

\subsection{Main result of Csordas et.~al.~}

We have an analogue of the main result of \cite{csv} and \cite{stopple}, which can be used to give lower bounds on $\Lambda$.

\begin{lem}\label{lem:stopple-bound-on-Lambda}
Let $(q, D)$ be a good pair and suppose the zeros of $\Xi_0(x, \chi_D)$ are simple. Let the positive zeros of $\Xi_0(z, \chi_D)$ be denoted $\gamma_1, \gamma_2, \ldots$ as described in \Cref{sec:notation-for-zeros}. Define the quantity
\begin{equation}\label{eq:def-G}
G\ = \
\sum_{j = 2}^\infty
\frac{2}{(\gamma_{1} - \gamma_j)^2}.
\end{equation}
Then if $5 \gamma_1^2 G < 1$, we have
\begin{equation}\label{eq:stopple-bound-on-Lambda}
\Lambda_D
\ > \
\frac{(1 - 5 \gamma_1^2 G)^{4/5}-1}{8G}.
\end{equation}
\end{lem}

\begin{proof}
This is a direct generalization of \cite[Theorem 1]{stopple}. In \cite{stopple}, the quantity $G$ has the same form except the sum is over the zeros of a number field $L$-function.\footnote{The quantity defined in \eqref{eq:def-G} is actually called ``$g(0)$'' in \cite{stopple}. However, in this paper, we use $g$ for the genus of the hyperelliptic curve defined by $D$.} The condition ($5\gamma_1^2 G < 1$) is the same as in \cite{stopple}, and the conclusion (a lower bound on $\Lambda_D$) is the same. The method of proof uses differential equations governing the motion of the zero $\gamma_1$ as $t$ changes to find a time $t < 0$ when $\gamma_1$ coalesces with $-\gamma_1$.
\end{proof}

\subsection{Low-lying zeros and connections with random matrix theory}\label{subsec:lowlyingzerosandconnections}

We now show that the condition $5 \gamma_1^2 G$ in \Cref{lem:stopple-bound-on-Lambda} does occur in certain families, using connections to random matrix theory. We begin by analyzing \eqref{eq:def-G}.

We assume the zeros of $\Xi_0(x, \chi_D)$ are simple, so by the discussion in \Cref{sec:notation-for-zeros}, we can write the first $g$ positive zeros as $0 < \gamma_1 < \cdots < \gamma_g < \pi$. Then all of the zeros of $\Xi_0$ are given by \Cref{eq:all-zeros-gamma}. Using this fact, we can write \eqref{eq:def-G} as
\begin{equation}
G
=
\sum_{\epsilon\in\{\pm 1\}}
\sum_{j=1}^g
\sum_{\ell \in \bbZ}{}' \
\frac{2}{[\gamma_1 - (\epsilon \gamma_j + 2\pi \ell)]^2},
\end{equation}
where the prime mark ($'$) means we omit the two terms $(\epsilon, j, \ell) = (\pm 1, 1, 0)$. Using the identity $\sum_{n \in \bbZ} (n + \alpha)^{-2} = \pi\csc^2 \pi \alpha$, after some algebraic manipulations, we obtain
\begin{equation}\label{eq:G-finite-sum}
G\  = \
\frac{1}{6}
-
\frac{1}{2\gamma_1^2}
+
\frac{1}{2} \csc^2 \gamma_1
+
\frac{1}{2}
\sum_{\epsilon \in \{\pm 1\}} \sum_{j=2}^g
\csc^2
\left(
\frac{\gamma_1 + \epsilon \gamma_j}{2}
\right).
\end{equation}
Observe that the sum on the right is now a finite sum.

With some work, we can determine sufficient conditions for $5 \gamma_1^2 G$, which allows us to apply \Cref{lem:stopple-bound-on-Lambda}.




\begin{lem}\label{lem:fn-field-crude-conditions}
Let $(q, D)$ be a good pair and let $\gamma_1, \ldots, \gamma_g$ be the zeros of $\Xi_0(x, \chi_D)$ in $[0, \pi]$. Assume the zeros are simple so that $0 < \gamma_1 < \cdots < \gamma_g < \pi$. Suppose the following conditions hold
\begin{itemize}
\item $g \geq 13$,
\item $\left(\frac{g}{\pi} \gamma_1\right)^2 \leq \frac{1}{500 g}$,
\item $\frac{1}{2} \leq \frac{g}{\pi} \gamma_2 \leq 2$.
\end{itemize}
Then $5\gamma_1^2G < 1$ (where $G$ is defined in \eqref{eq:G-finite-sum}).
\end{lem}

Before we present the proof, we make a few observations. The quantities $\tilde\gamma_j := \frac{g}{\pi} \gamma_j$ are rescalings of the zeros. Since $0 < \tilde \gamma_1 < \cdots < \tilde \gamma_g < g$, the normalized zeros $\tilde \gamma_1, \tilde \gamma_2, \ldots$ on average have unit spacing on the positive real line. Thus the condition $\tilde \gamma_1^2 \leq \frac{1}{500g}$ says that the first zero is unusually small, while the condition $\frac{1}{2} \leq \tilde\gamma_2 \leq 2$ says that second zero is around where it is ``expected'' to be.

These conditions (along with $g \geq 13$) are very crude and the constants can easily be improved with some work. However, our focus is not on the optimum, but the fact that such a statement as the lemma exists.

\begin{proof}
This argument is technical, and relies on bounds for the function $\csc^2 x$. In particular, for $|x| \leq \frac{1}{2}$,
\begin{align}
\label{eq:csc-bound-weak}
\csc^2 x
&\leq
\frac{1.1}{x^2}
,
\\
\label{eq:csc-bound-strong}
\csc^2 x
&\leq
\frac{1}{x^2}
+
0.36
.
\end{align}

Next, we take the expression \eqref{eq:G-finite-sum} for $G$ and break it into two parts by writing
\begin{equation}\label{eq:G-break-up}
G
\ = \
\frac{1}{6}
+
S_{I}
+
S_{II}
,
\end{equation}
where
\begin{align}
\begin{split}
S_I
&\ = \
\frac{1}{2} \csc^2 \gamma_1
-
\frac{1}{2\gamma_1^2}
\\
S_{II}
&\ = \
\frac{1}{2}
\sum_{\epsilon \in \{\pm 1\}} \sum_{j=2}^g
\csc^2
\left(
\frac{\gamma_1 + \epsilon \gamma_j}{2}
\right)
.
\end{split}
\end{align}
Using the bound \eqref{eq:csc-bound-strong} and our assumption on $\gamma_1$, we have
\begin{equation}\label{eq:S-I-bound}
S_I \ \leq \ 0.18.
\end{equation}

Next we bound $S_{II}$. The idea will be to bound the sum by the maximum term, i.e.,
\begin{equation}
S_{II}
\ = \
\frac{1}{2} \cdot 2 \cdot (g-1)
\cdot
\max_{
 \substack{
 \epsilon \in \{\pm 1\}
 \\
 2 \leq j \leq g
 }
}
\left\{
\csc^2
\left(
\frac{\gamma_1 + \epsilon \gamma_j}{2}
\right)
\right\}.
\end{equation}
Notice that $\csc^2 x$ is large when $x$ is near a multiple of $\pi$. The choice of $(\epsilon, j) \in \{ \pm 1\} \times \{2, \ldots, g\}$ that minimizes the distance between $\frac{\gamma_1 + \epsilon\gamma_j}{2}$ and $0$ is $(\epsilon, j) = (-1, 2)$. That distance is
\begin{equation}\label{eq:gamma1-gamma2-difference}
\left|
0
-
\frac{\gamma_1 - \gamma_2}{2}
\right|
\ = \
\frac{\gamma_2 - \gamma_1}{2}
\ \leq\
\frac{\gamma_2}{2}
\ \leq\
\frac{\pi}{g}
\ \leq\
\frac{\pi}{13}.
\end{equation}

The choice of $(\epsilon, j)$ that minimizes the distance between $\frac{\gamma_1 + \epsilon\gamma_j}{2}$ and $\pi$ is $(\epsilon, j) = (1, g)$. That distance is
\begin{equation}
\left|
\pi
-
\frac{\gamma_1 + \gamma_g}{2}
\right|
\ = \
\pi
-
\frac{\gamma_1 + \gamma_g}{2}
\ \geq\
\pi
-
\frac{1 + \pi}{2}
\ \geq\
\frac{\pi}{2} - \frac{1}{2}.
\end{equation} Note that it suffices to obtain a lower bound on the absolute value of the difference, as if it were large than it would be closer to a different multiple of $\pi$.

The choice of $(\epsilon, j)$ that minimizes the distance between $\frac{\gamma_1 + \epsilon\gamma_j}{2}$ and $-\pi$ is $(\epsilon, j) = (-1, g)$. That distance is
\begin{equation}
\left|
-\pi
-
\frac{\gamma_1 - \gamma_g}{2}
\right|
\ = \
\pi
-
\frac{-\gamma_1 + \gamma_g}{2}
\ \geq\
\pi
-
\frac{0 + \pi}{2}
\ \geq\
\frac{\pi}{2}.
\end{equation}
It follows that $\csc^2
\left(
\frac{\gamma_1 + \epsilon \gamma_j}{2}
\right)$ is maximized at $(\epsilon, j) = (-1, 2)$, so
\begin{equation}\label{eq:S-II-initial-bound}
S_{II}
\ \leq\
\frac{1}{2}
\sum_{\epsilon \in \{\pm 1\}} \sum_{j=2}^g
\csc^2
\left(
\frac{\gamma_1 - \gamma_2}{2}
\right)
\ \leq\
g
\csc^2
\left(
\frac{\gamma_1 - \gamma_2}{2}
\right)
.
\end{equation}

As shown in \eqref{eq:gamma1-gamma2-difference}, we have $\left|\frac{\gamma_1 - \gamma_2}{2}\right| \leq \frac{1}{2}$. Thus, combining \eqref{eq:csc-bound-strong} and \eqref{eq:S-II-initial-bound} yields
\begin{equation}\label{eq:S-II-bound}
S_{II}
\ \leq\
1.1g\left(\frac{2}{\gamma_1 - \gamma_2}\right)^2
\ = \
\frac{4.4g}{\gamma_2^2} \left(\frac{1}{1 - \gamma_1/\gamma_2}\right)^2
\ \leq\
\frac{18g}{\gamma_2^2}.
\end{equation}
By combining \eqref{eq:G-break-up}, \eqref{eq:S-I-bound}, and \eqref{eq:S-II-bound}, we arrive at
%
\begin{equation}\label{eq:5gammaG-bound}
5\gamma_1^2 G
\ \leq\
1.8 \gamma_1^2
+
90g \cdot \frac{\gamma_1^2}{\gamma_2^2}.
\end{equation}

By using $\tilde \gamma_2 \geq \frac{1}{2}$, we have
\begin{equation}
90g \cdot \frac{\gamma_1^2}{\gamma_2^2}
\ = \
90g \cdot \frac{\tilde \gamma_1^2}{\tilde \gamma_2^2}
\ \leq\
450 g \tilde\gamma_1^2
\ \leq\
0.9,
\end{equation}
where we use $\tilde\gamma_1^2 \leq \frac{1}{500g}$ at the end. Then $1.8 \tilde \gamma_1^2 \leq 1.8 \cdot\frac{1}{500 \cdot 13} \leq 0.0003$, so $5\gamma_1^2G < 1$. \end{proof}

As remarked by Stopple, the expression on the right hand side of \eqref{eq:stopple-bound-on-Lambda} has the power series expansion
\begin{equation}
\frac{(1 - 5 \gamma_1^2 G)^{4/5}-1}{8G}
\ = \
-\frac{1}{2} \gamma_1^2
\left(
1
+
\frac{\gamma_1^2 G}{2}
+
O(\gamma_1^4G^2)
\right).
\end{equation}
Thus the smaller the first zero $\gamma_1$ is, the better the lower bound on $\Lambda$ given by \Cref{lem:fn-field-crude-conditions} is.

We discuss an interpretation of the above. Since our Newman conjectures vary over families, we write $\tilde \gamma_j(D)$ and $g(D)$ to remind ourselves of dependence on $D$.

\begin{cor}
Let $\mathcal F$ be a family of polynomials $D$ belonging to good pairs $(q, D)$. Suppose there exists a sequence $D_1, D_2, \ldots$ in $\mathcal F$ such that $\gamma_1(D_n) \to 0$ as $n \to \infty$ and for all $n$,

\begin{itemize}
\item
$g(D_n) \geq 13$
\item
$\tilde \gamma_1(D_n)^2 \leq \frac{1}{500 g(D_n)}$
\item
$\frac{1}{2} \leq \tilde \gamma_2(D_n) \leq 2$.
\end{itemize}

Then $\Lambda_{D_n} \to 0$ as $n \to \infty$, so Newman's conjecture is true for the family $\mathcal F$. \end{cor}

The conditions above essentially say that there is a set of curves in our family where the first zero is unusually small and the second zero is on the order of its expected value. For many families with $g$ and $q$ tending to infinity, this is known due to work of Katz and Sarnak \cite{ks-article, ks-book}.


\appendix

\section{Example of $\Xi_t(x, \chi_D)$ in function fields and the role of the backwards heat equation.}\label{sec:backwards-heat-eq-example}

For
\begin{equation}
D \ = \  T^5 + T^4 + T^3 + 2T + 2 \in \bbF_5[T]
,
\end{equation}
we have
\begin{equation}\label{eq:sample-xi-t}
\Xi_t(x, \chi_D)
\ = \
10 e^{4t} \cos 2x -2 \sqrt{5} e^t \cos x  -1
.
\end{equation}
For any $t$, we observe that $e^{2ix} \cdot \Xi_t(x, \chi_D)$ is a fourth degree polynomial in $e^{ix}$. Thus $\Xi_t(x, \chi_D)$ must have exactly four zeros with $\Re(x) \in [0, 2\pi)$.

\Cref{fig:example-plots} shows plots of $\Xi_t(x, \chi_D)$ for various times $t$. Observe that as we move backwards in time, the peaks get smaller. Because $\Xi_t(x, \chi_D)$ solves the backwards heat equation, the ``flattening'' of the function behaves like the diffusion of heat.

\begin{figure}[h!]
        \centering
        \begin{subfigure}[b]{0.40\textwidth}
                \centering
                \includegraphics[width=\textwidth]{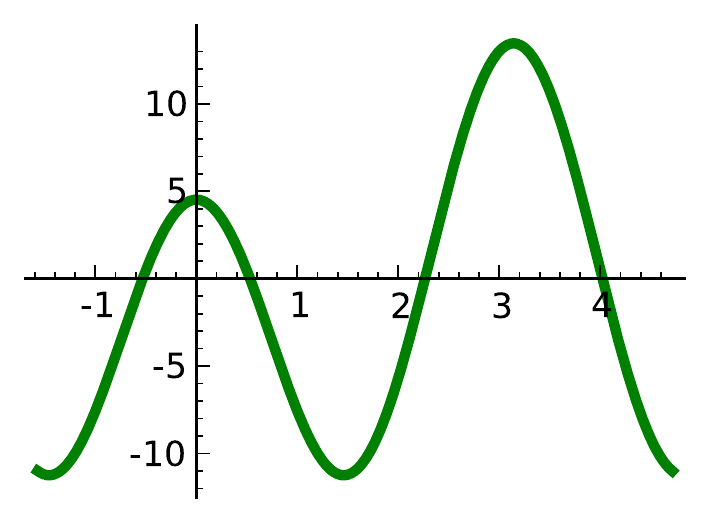}
                \caption{$t = 0$}
                \label{}
        \end{subfigure}%
\qquad\qquad
        ~ 
        \begin{subfigure}[b]{0.40\textwidth}
                \centering
                \includegraphics[width=\textwidth]{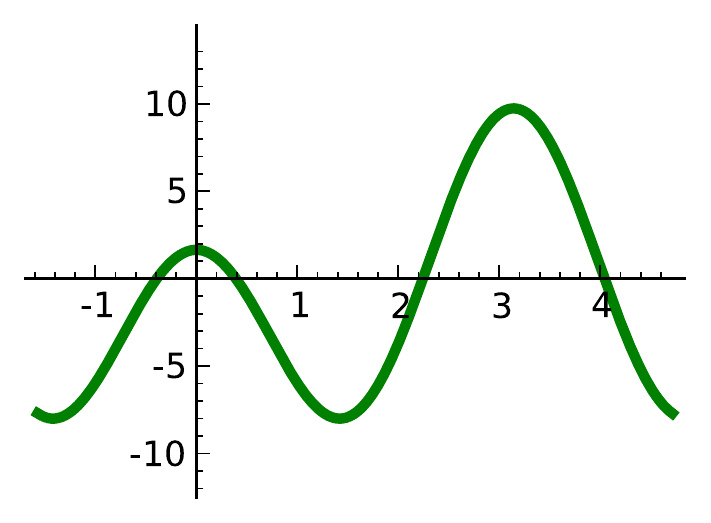}
                \caption{$t = -0.1$}
                \label{}
        \end{subfigure}%
\\
        ~ 
        \begin{subfigure}[b]{0.40\textwidth}
                \centering
                \includegraphics[width=\textwidth]{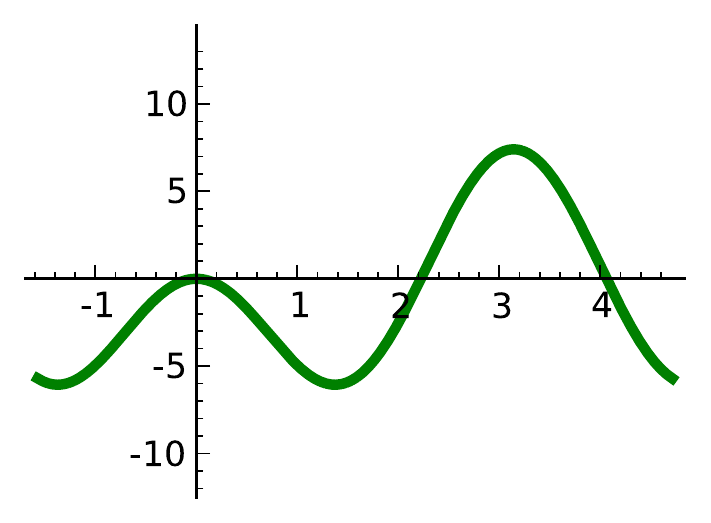}
                \caption{$t \approx -0.189$}
                \label{}
        \end{subfigure}
\qquad\qquad
        ~ 
        \begin{subfigure}[b]{0.40\textwidth}
                \centering
                \includegraphics[width=\textwidth]{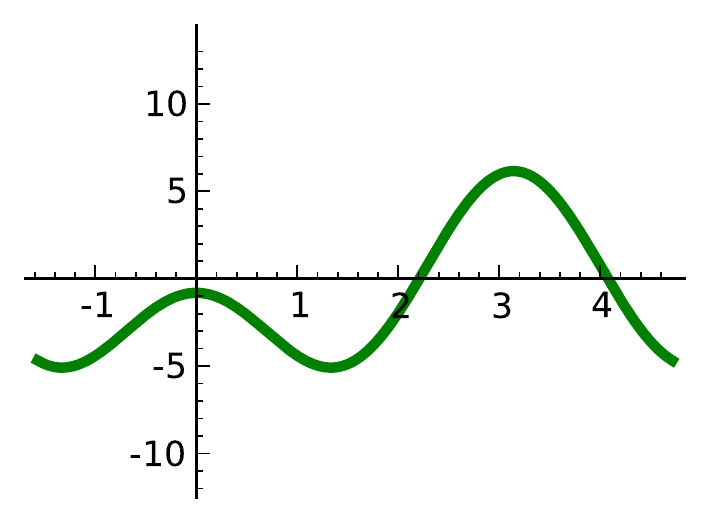}
                \caption{$t = -0.25$}
                \label{}
        \end{subfigure}
        \caption{Plots of $\Xi_t(x, \chi_D)$ for different $t
$.}
        \label{fig:example-plots}
\end{figure}

As we decrease $t$, the two zeros on the left move towards each other, until they coalesce at $t \approx -0.189$. If we keep going further back in time, these two zeros ``pop off'' the real line. For instance, at $t = -0.25$, the function has zeros at $x \approx \pm 0.152i$.

The time when the zeros coalesce ($t \approx -0.189$) is the De Bruijn--Newman constant $\Lambda_D$ for this $D$. It is the largest real solution to $\Xi_t(0, \chi_D) = 0$. From \eqref{eq:sample-xi-t}, we see that $\Lambda_D$ is the logarithm of the root of a fourth degree polynomial.

\section{Numerical calculations}\label{sec:numerical}

If $\Xi_t(\cdotwild, \chi_D)$ has a zero at $x = 0$, then it has a double zero there by evenness. Thus, by using \Cref{lem:double-zero-lower-bound-fnfield}, we see that a solution $t$ to $\Xi_t(0, \chi_D) = 0$ is a lower bound for $\Lambda_D$. We have
\begin{equation}
\Xi_t(0, \chi_D)
\ = \
\Phi_0 + 2\sum_{n=1}^g \Phi_n (e^{t})^{n^2}
,
\end{equation}
which is a polynomial in $e^t$ of degree $g^2$. As $g$ increases, it becomes harder to find the exact roots of this polynomial, but we may still proceed numerically. This gives us a method to quickly find lower bounds of $\Lambda_D$ for $D$.
%
%

For $q = 3$, this method produces the lower bounds given in \Cref{fig:table-lower-bounds-q-3} and \Cref{fig:table-lower-bounds-q-3-polys}.

\begin{figure}[h!]
        \centering
\[
    \begin{array}{l|l|c}
    g & (c_0, \ldots, c_g) & \text{lower bound on } \Lambda_D
\\ \hline
    1 & (1,-3)                  & -1.44 \cdot 10^{-1}
\\ \hline
    2 & (1,-3,5)                 & -5.28 \cdot 10^{-2}
\\ \hline
    3 & (1,-1,1,-7)                 & -1.26 \cdot 10^{-2}
\\ \hline
    4 & (1,-3,9,-23,39)                  & -1.05 \cdot 10^{-3}
\\ \hline
    5 & (1,-3,5,-3,-11,27)                  & -1.23        \cdot 10^{-4}
\\ \hline
    6 & (1,-1,3,-7,5,-13,11)                  & -3.02        \cdot 10^{-5}
\\ \hline
    7 & (1,1,5,3,1,-15,-51,-101)                  & -1.28        \cdot 10^{-5}
    \end{array}
\]

        \caption{Lower bounds on $\Lambda_D$ for certain $D \in \bbF_3[T]$. The values $c_0, \ldots, c_g$ are the coefficients of the $L$-function as in \eqref{eq:l-fn-sum-over-q-s} and \eqref{eq:l-fn-coefficients}.}
        \label{fig:table-lower-bounds-q-3}
\end{figure}

\begin{figure}[h!]
        \centering
\[
    \begin{array}{l|l}
    g & D
\\ \hline
    1 & T^3+2T+1
\\ \hline
    2 & T^5+T^3+T+1
\\ \hline
    3 & T^7+2T^5+T^3+2T^2+2T+2
\\ \hline
    4 & T^9+T^6+T^4+T^3+T^2+T+1
\\ \hline
    5 & T^{11}+2T^9+T^8+2T^7+2T^6+2T^5+2T^4+T^2+2T+1
\\ \hline
    6 & T^{13}+2T^{11}+T^{10}+2T^7+2T^6+2T^4+T^3+2T+1
\\ \hline
    7 & T^{15}+2T^{14}+2T^9+T^8+2T^6+T^3+2T^2+T+2
    \end{array}
\]

        \caption{Polynomials in $\bbF_3[T]$ used in \Cref{fig:table-lower-bounds-q-3}.}
        \label{fig:table-lower-bounds-q-3-polys}
\end{figure}

The above supports the claim that
\[
\lim_{g \to \infty}
\sup_{
 \substack{D \in \bbF_3[T]\\ \deg D \leq 2g + 1}
}
\Lambda_D
=0,
\]
which supports Newman's conjecture for fixed $q$.


\

\end{document}